\documentclass[leqno]{amsart}
\usepackage{amsmath}
\usepackage{amssymb}
\usepackage{amsthm}
\usepackage{enumerate}
\usepackage[mathscr]{eucal}
\theoremstyle{plain}
\usepackage{tikz}
\newtheorem{theorem}{Theorem}[section]
\newtheorem{lemma}[theorem]{Lemma}
\newtheorem{prop}[theorem]{Proposition}
\theoremstyle{definition}
\newtheorem{definition}[theorem]{Definition}
\newtheorem{remark}[theorem]{Remark}

\newtheorem{example}[theorem]{Example}

\newtheorem{cor}[theorem]{Corollary}
\theoremstyle{remark}

\begin{document}
\title[On Isosceles orthogonality and some geometric constants] {On isosceles orthogonality and some geometric constants in a normed space}
\author[Sain,  Ghosh and  Paul  ]{Debmalya Sain, Souvik Ghosh and Kallol Paul }

\newcommand{\acr}{\newline\indent}

\address[Sain]{Departmento de Analisis Matematico\\ Universidad de Granada\\ Spain. }
\email{saindebmalya@gmail.com}

\address[Ghosh]{Department of Mathematics\\ Jadavpur University\\ Kolkata 700032\\ West Bengal\\ INDIA}
\email{sghosh0019@gmail.com}

\address[Paul]{Department of Mathematics\\ Jadavpur University\\ Kolkata 700032\\ West Bengal\\ INDIA}
\email{kalloldada@gmail.com}

\thanks{The research of Dr. Debmalya Sain is sponsored by a Maria Zambrano postdoctoral grant under the mentorship of Prof Miguel Martin. The second author would like to thank  CSIR, Govt. of India, for the financial support in the form of Junior Research Fellowship under the mentorship of Prof. Kallol Paul.} 


\maketitle

\begin{abstract}
  We study the James constant $J(\mathbb{X})$, an important geometric quantity associated with a normed space $ \mathbb{X} $, and explore its connection with isosceles orthogonality $ \perp_I. $ The James constant is defined as $J(\mathbb{X}) := \sup\{\min \{\|x+y\|, \|x-y\|\}: x, y \in \mathbb{X},~ \|x\|=\|y\|=1 \}.$ We prove that if $J(\mathbb{X})$ is attained for unit vectors $x, y \in \mathbb{X},$ then $x\perp_I y.$ We also show that if $\mathbb{X}$ is a two-dimensional polyhedral Banach space then $J(\mathbb{X})$ is always attained at an extreme point $z$ of the unit ball of $\mathbb{X},$ so that $J(\mathbb{X}) = \|z+y\| = \|z-y\|,$ where $ \| y \| = 1 $ and $z\perp_I y.$ This helps us to explicitly compute the James constant of a two-dimensional polyhedral Banach space in an efficient way. We further study some related problems with reference to several other geometric constants in a normed space.
\end{abstract}

\section{Introduction}

There are various geometric constants associated with a normed space, which are useful  towards a quantitative understanding of the  geometry of the space and also play an important role in the study of  some other related problems of functional analysis. The James constant is one of the most prominent geometric constants associated with the space, which measures the ``non-squareness'' of the unit ball of a normed space. Our motivation behind this article is to illustrate the central role played by isosceles orthogonality, a natural generalization of the usual orthogonality in an inner product space, in studying various geometric constants, including the James constant. Before proceeding further, let us fix the notations and the terminologies.\\

 Let  $\mathbb{X}, \mathbb{Y}$ denote  real normed spaces. Let $B_\mathbb{X} = \{ x \in \mathbb{X} : \|x\|\leq 1\} $ and $S_\mathbb{X} = \{x\in \mathbb{X}: \|x\|=1\}$ denote  the unit ball and the unit sphere of $\mathbb{X},$ respectively.  For a non-empty convex subset $S$ of $\mathbb{X},$ an element $z \in S$  is said to be an extreme point of $S$ if $ z = (1-t)x + ty $ for some $ t \in (0,1) $ and $x,y \in S$ implies that $ x=y=z.$ The set of all extreme points of $B_\mathbb{X}$ is denoted by $E_{\mathbb{X}}.$  A normed space $\mathbb{X}$ is said to be  strictly convex if $ E_{\mathbb{X}} = S_{\mathbb{X}}.$
  An element  $x \in \mathbb{X}$ is said to be isosceles orthogonal \cite{J}  to an element $y \in \mathbb{X} $, denoted as $x\perp_I y$,  if 
  $ \|x + y\|= \|x - y\|.$ Geometrically it means that the length of the two diagonal vectors $ \|x +y\|$ and $ \|x -y\|$ of the parallelogram formed by two vectors  $x$ and $y$  are equal. We refer the readers to \cite{A, AMW, JLW} for more information related to this topic. An element $x \in \mathbb{X}$ is said to be approximate isosceles orthogonal \cite{CW} to $y$   if for $\epsilon \in [0, 1), $ $|\|x+y\|^2 - \|x-y\|^2| \leq 4\epsilon \|x\|\|y\|,$ and is written as $ x \perp_I^{\epsilon} y.$ Note that approximate isosceles orthogonality is symmetric, and therefore, so is exact isosceles orthogonality. 
  
  \smallskip
  
  We now mention the definitions of the following geometric constants, to be studied throughout this paper.
  \begin{definition} \cite{GL}
  	Let $\mathbb{X}$ be a normed space. 
  	\begin{enumerate}[(i)]
  		\item The James constant, denoted by $J(\mathbb{X}), $ is defined as 
  		\[ J(\mathbb{X}) = \sup\Big \{\min \big \{\|x+y\|, \|x-y\| \big \} : x, y \in S_\mathbb{X} \Big\}.	\]
  		\item For $ x \in S_\mathbb{X},$ the local James constant, denoted by $\beta(x),$ is defined as 
  		\[ \beta(x) = \sup \Big \{\min \big \{\|x+y\|, \|x-y\| \big \} : y\in S_\mathbb{X}\Big\}.\]  
  		
  		\item The Sch\"affer constant, denoted by $ S(\mathbb{X}), $ is defined as 
  		\[S(\mathbb{X}) =  \inf \Big \{\max \big \{\|x+y\|, \|x-y\| \big \} : x, y\in S_\mathbb{X} \Big \}.\]
  		
  		\item For each $ x \in S_\mathbb{X},$ the local Sch\"affer constant at $ x $, denoted by $ \alpha(\mathbb{X}), $ is defined as 
  		 \[ \alpha(x) =  \inf \Big \{\max \big \{\|x+y\|, \|x-y\| \big \} : y \in S_\mathbb{X}\Big \}.\] 
  \end{enumerate}
  	\end{definition}
   	 We note from \cite{GL} that for a given normed space $\mathbb{X},$ $\sqrt{2} \leq J(\mathbb{X}) \leq 2.$ Moreover, $ \mathbb{X} $ is said to be uniformly non-square if and only if $ J(\mathbb{X}) < 2. $ It is also known that $ J(\mathbb{X}) = \sqrt{2} $ whenever $\mathbb{X}$ is an inner product space but the converse is not true, in general. In \cite{KST}, the authors studied the normed spaces with James constant $\sqrt{2}.$
   	 
  Generalizations of the notions of the James constant and the local James constant, were introduced in \cite{LSL} in the following way.
 For $ \lambda \in (0, 1),$ the generalized James constant, denoted by $J(\lambda, \mathbb{X}),$ is defined as 
  \[ J(\lambda, \mathbb{X}) = \sup \Big \{\min \big \{\|\lambda x + (1-\lambda) y\|, \| \lambda x - (1-\lambda) y\| \big \} : x, y \in S_{\mathbb{X}} \Big\}\]  and for  $x\in S_\mathbb{X},$ the generalized local James constant, denoted by $ \beta(\lambda, x), $ is defined as 
 \[\beta(\lambda, x) = \sup \Big \{\min \big \{\|\lambda x+(1-\lambda)y\|, \|\lambda x - (1-\lambda)y\|\big \} :  y \in S_\mathbb{X} \Big \}.\]
 We also need two other well-known geometric constants,  modulus of smoothness and modulus of convexity, which are denoted by $ \rho_{\mathbb{X}}(\epsilon) $ and $ \delta_{\mathbb{X}}(\epsilon), $ respectively, and are defined as 
  \[\rho_{\mathbb{X}}(\epsilon) = \sup \Big\{ 1 - \frac{\|x + y\|}{2} : x,y \in S_{\mathbb{X}}, \|x - y \| \leq \epsilon\Big \},\] 
 \[ \delta_{\mathbb{X}}(\epsilon) = \inf \Big\{ 1 - \frac{\|x + y\|}{2} : x,y \in S_{\mathbb{X}}, \|x - y \| \geq \epsilon\Big \},\]  
 where $ \epsilon \in [0,2].$  We note from \cite[Cor. 5]{L} that $  \delta_{\mathbb{X}}$ is a continuous function on $[0,2)$ whereas from \cite{WY}, $\rho_{\mathbb{X}}$ is continuous  on $[0, 2].$ The modulus of smoothness is also defined as:
 	\[ \rho'_{{\mathbb{X}}}(\epsilon) = \sup_{x, y\in S_\mathbb{X}}\Big\{\frac{\|x + \epsilon y\| + \|x - \epsilon y\|}{2}-1\Big\}.\]
 	or (equivalently)
 	\[  \rho'_{\mathbb{X}}(\epsilon) = \Big\{\frac{\|x + y\| + \|x - y\|}{2}-1 : \|x\| =1, \|y\| \leq \epsilon \Big\}.\]
  Observe that $\rho'_{{\mathbb{X}}}(\epsilon)$  is not equivalent to $\rho_{\mathbb{X}}(\epsilon)$ ( see \cite[Th. 1]{B} ). 
  
  \smallskip
  
   Given any $x, y \in \mathbb{X},$ we denote by $[x, y\rangle$ the ray passing through $y$ and starting from $x,$ i.e., $ [x, y\rangle = \{ (1-t)x + ty : t \geq 0\}$ and $[x, y]$ denotes the closed convex line segment between $x$ and $y,$ i.e., $[x, y] = \{(1-t) x + ty : 0\leq t \leq 1\}.$ Another important concept to be used in this paper is that of orientation. Following \cite{BFS},  we say that $x$ precedes $y$ in a two-dimensional Banach space $ \mathbb{X}$, if  
  $x_1y_2 - x_2y_1 > 0,$ where $x = (x_1, y_1), y=(y_1, y_2) \in \mathbb{X}$ and  in this case we write that $ x \prec y.$ Of course, here $ \mathbb{X} $ is identified with $ \mathbb{R}^2 $ in the obvious way. We note from  \cite[Cor.2.4]{JLW} that for any $x \in S_\mathbb{X}$ there exists  a unique (except for the sign) $y \in S_\mathbb{X}$ such that $x \perp_I y.$ In particular, whenever it is given that $ x \perp_I y,$ without loss of generality we can assume that $ -y \prec x \prec y.$ 
   We also consider the attainment set $ M_{J(\mathbb{X})} $ of the James constant: 
   \[M_{J(\mathbb{X})} = \{(x, y) \in S_\mathbb{X} \times S_\mathbb{X} : \min\{\|x+y\|, \|x-y\|\} = J(\mathbb{X})\}.
   \]
 When $\mathbb{X}$ is finite-dimensional it is easy to see that $M_{J(\mathbb{X})} \neq \emptyset.$ 
 
 \smallskip
 We explore the attainment problem for the generalized James constant and also study its converse. We illustrate the crucial role played by isosceles orthogonality in the whole scheme of things. In two-dimensional polyhedral Banach spaces, we make an observation which is computationally effective for finding the values of the James constant in each case. We also study approximate isosceles orthogonality from a geometric point of view and discuss its connections with the modulus of convexity.\\

 We end this section by mentioning  the following  known results, which are essential in our present work.
 \begin{lemma}\label{monotonicity}\cite[Prop. 31]{MSW}
 Let $\mathbb{X}$ be a two-dimensional Banach space. Let $x, y, z \not = \theta,$ $x\not= z,$ with $[0, y\rangle$ lying in between $[0, x\rangle$ and $[0, z\rangle,$ and suppose that $\|y\| = \|z\|.$ Then $\|x-y\| \leq \|x-z\|.$ In particular, if $\mathbb{X}$ is strictly convex, then we always have strict inequality. 
 \end{lemma}

 \begin{lemma}\label{lemma}\cite[Lemma 2.2]{GL}
 		Let $\mathbb{X}$ be a two-dimensional Banach space and let $x \in S_\mathbb{X}.$ Then there exists a unique $y \in S_\mathbb{X}$ such that $\alpha(x) = \beta(x) = \|x+y\| = \|x-y\|.$
 	\end{lemma}

 \begin{theorem}\label{sup}\cite[Th. 3.3]{GL} 
 Let $\mathbb{X}$ be a normed space. Then \\
  \[ J(\mathbb{X}) = \sup\big\{\epsilon : \epsilon < 2-2\delta_{\mathbb{X}}(\epsilon)\}.\]
 \end{theorem}
 
 \begin{prop}\label{equal} \cite[Prop. 2.8]{GL}
 	Let $\mathbb{X}$ be two-dimensional Banach space. If $S_\mathbb{X}$ is affinely homeomorphic to a convex symmetric body in the two-dimensional Euclidean space $\mathbb{R}^2$  which is invariant under a rotation of $\frac{\pi}{4}$, then $J(\mathbb{X}) = \sqrt{2}.$
 \end{prop}
 
 \begin{theorem}\label{uniqueness}\cite[Th. 2.3]{JLW}
 Let $\mathbb{X}$ be a two-dimensional Banach space and let $x \in \mathbb{X}$ be non-zero. Then for each number $0 \leq r \leq \|x\|,$ there exists a unique $y \in \in rS_{\mathbb{X}}$ such that $x \perp_I y.$ \\
 	 Moreover, if $\mathbb{X}$ is strictly convex then for each $r \in [0, +\infty),$ there exists a unique $y \in rS_{\mathbb{X}}$ such that $x \perp_I y.$
 \end{theorem}

 \section{Main Results}

 In \cite{GL}, Gao and Lau proved that in a two-dimensional Banach space $\mathbb{X}$  if $x, y \in S_{\mathbb{X}}$ are such that $x \perp_{I} y,$ then $\beta(x) = \beta(y) = \|x - y\| = \|x + y \|.$ We begin with a proposition by establishing a similar result in the case of the generalized local James constant $\beta(\lambda, x),$ from which the above result follows directly as a particular case ($\lambda = \frac{1}{2}$).

 \begin{prop}\label{gen}
 Let $ \mathbb{X}$ be a two-dimensional Banach space and  $ x, y \in S_{\mathbb{X}}.$ If $x \perp_{I} (\frac{1-\lambda}{\lambda}) y,$ where $ \lambda \in (0, 1),$ then $ \beta(\lambda, x) = \|\lambda x + (1 - \lambda)y\| =  \|\lambda x - (1 - \lambda)y\|. $ 
 
 \end{prop}

 \begin{proof}
 Let  $ x \perp_{I}(\frac{1-\lambda}{\lambda}) y.$ Then we get,  $ \|\lambda x + (1 - \lambda)y\| =  \|\lambda x - (1 - \lambda)y\|.$ Clearly, for any $ z\not= \pm y$ we have $(1 - \lambda) z  \not= \pm (1 - \lambda)y.$ Consider the following four sets : 
   \begin{align*}
       C_1 = \{(1-\lambda) \frac{(1-t)x + t y}{\|(1-t)x + t y\|} : 0 \leq t \leq 1\},\\
       C_2 = \{(1-\lambda) \frac{(1-t)y - t x}{\|(1-t)y - t x\|} : 0 \leq t \leq 1\},\\
       C_3 = \{(1-\lambda) \frac{-(1-t)x - t y}{\|-(1-t)x - t y\|} : 0 \leq t \leq 1\},\\
       C_4 = \{(1-\lambda) \frac{-(1-t)y + t x}{\|-(1-t)y + t x\|} : 0 \leq t \leq 1\},
     \end{align*}
  whose union is the circle of radius $|1-\lambda|$ and the sets $C_i$ intersect only at $ \pm (1-\lambda)x, \pm (1-\lambda)y.$ 
  Observe  that for any $z\in S_\mathbb{X},$ we have $(1 - \lambda)z \in C_i,$ for some $ i,~ 1\leq i \leq 4. $ Let us assume that $ (1 - \lambda)z \in C_1.$ Then applying Lemma \ref{monotonicity} it is straightforward to observe that $\|\lambda x + (1 - \lambda)z\| \geq \|\lambda x + (1 - \lambda)y\| $ whereas $\|\lambda x - (1 - \lambda) z\| \leq \|\lambda x - (1 - \lambda)y\|.$  Therefore, we obtain, 
  $ \min\{\|\lambda x - (1 - \lambda)y\|, \|\lambda x + (1 - \lambda)y\| \} =  \|\lambda x - (1 - \lambda)y\| \geq 
  \|\lambda x - (1 - \lambda)z\|\geq  \min\{\|\lambda x - (1 - \lambda)z\|, \|\lambda x + (1 - \lambda)z\|\}.$
  If  $ (1-\lambda)z \in C_i,$ for some $i\in \{2, 3, 4\}$ then we can  proceed similarly to conclude that 
$ \min\{\|\lambda x - (1 - \lambda)y\|, \|\lambda x + (1 - \lambda)y\| \} \geq \min\{\|\lambda x - (1 - \lambda)z\|, \|\lambda x + (1 - \lambda)z\|\}.$ As $z \in S_{\mathbb{X}}$ is arbitrary, we get $ \beta(\lambda, x) = \min\{\|\lambda x - (1 - \lambda)y\|,  \|\lambda x + (1 - \lambda)y\|\}=  \|\lambda x + (1 - \lambda)y\| =  \|\lambda x - (1 - \lambda)y\|. $  
 \end{proof}

 To determine the value of $J(\lambda, \mathbb{X})$ of a normed space $\mathbb{X},$ we observe the following:
 
 \begin{remark}
 Following Proposition \ref{gen}, it is easy to observe that for a given $ \lambda \in (0,1),$ 
\begin{eqnarray*}
 J(\lambda, \mathbb{X})& = & \sup\Big \{ \|\lambda x + (1 - \lambda)y\|: x,y \in S_{\mathbb{X}}, x \perp_{I}(\frac{1-\lambda}{\lambda}) y \Big \}\\
 &  = & \sup \Big \{\|\lambda x - (1 - \lambda)y\|: x,y \in S_{\mathbb{X}}, x \perp_{I}(\frac{1-\lambda}{\lambda}) y \Big \}.
\end{eqnarray*}
 Therefore, to find the generalized James constant $J(\lambda, \mathbb{X}),$ for a given $\lambda \in (0, 1),$ we only need to consider the subset $ \{(x, y) \in S_{\mathbb{X}} \times S_{\mathbb{X}} : x\perp_{I} (\frac{1-\lambda}{\lambda}) y\} \subseteq S_{\mathbb{X}} \times S_{\mathbb{X}}. $ 
 \end{remark}
 
 In the following theorem, we study the converse of Proposition \ref{gen}.

\begin{theorem}\label{general}
   Let $\mathbb{X}$ be a strictly convex normed space and $x \in S_{ \mathbb{X}},  \lambda \in (0, 1). $ If $ \beta(\lambda, x) = \min\{\|\lambda x + (1-\lambda) y\|, \| \lambda x - (1-\lambda) y\|\},$ for some  $y \in S_\mathbb{X}, $ then $ x \perp_{I} (\frac{1-\lambda}{\lambda})y.$ 
\end{theorem}
\begin{proof}
	Clearly $ x \neq \pm y.$ Since $x,y$ are linearly independent consider the two-dimensional subspace $\mathbb{Y} = span ~\{x, y\}.$  If possible let us assume that  $ x \not\perp_{I} (\frac{1-\lambda}{\lambda})y.$ Then either $ \|x + (\frac{1-\lambda}{\lambda})y\| > \|x - (\frac{1-\lambda}{\lambda})y\|$ or $ \|x - (\frac{1-\lambda}{\lambda})y\| > \|x + (\frac{1-\lambda}{\lambda})y\|.$ Without loss of generality we assume that $ \|x + (\frac{1-\lambda}{\lambda})y\| > \|x - (\frac{1-\lambda}{\lambda})y\|$ so that  $ \beta(\lambda, x) = \|\lambda x - (1-\lambda)y\|.$ Applying \textcolor{blue}{Theorem \ref{uniqueness}}, there exists a unique $ z \in S_{\mathbb{Y}}$ (except for the sign) such that $x\perp_{I} \frac{1-\lambda}{\lambda} z.$  Observe that either   \\
	 $(i)$ the ray $[0, (1-\lambda)y\rangle$ lies in between the rays $[0, \lambda x \rangle$ and $[0, (1- \lambda) z \rangle$ or \\
	$(ii)$ the ray $[0, (1-\lambda)y\rangle$ lies in between the rays $[0, \lambda x \rangle$ and $[0, -(1- \lambda )z \rangle.$\\
	Assume that $(i)$ holds. Since $ \lambda x, (1-\lambda y), (1- \lambda z) \neq \theta$ and $\|(1-\lambda)y\| = \|(1- \lambda)z\|$  applying Lemma \ref{monotonicity}, together with the assumption that $ \mathbb{X}$ is strictly convex, we conclude that $ \|\lambda x - (1- \lambda)y\| < \|\lambda x - (1- \lambda) z\|= \|\lambda x + (1-\lambda)z\|.$ This implies that $ \beta(\lambda, x) < \min\{\|\lambda x- (1-\lambda)z\|, \lambda x + (1-\lambda)z\|,$ a contradiction to the definition of $\beta(\lambda, x).$ If $(ii)$ holds then also we can proceed similarly. Thus we must have $ x \perp_{I} (\frac{1-\lambda}{\lambda})y.$  
\end{proof}

  It is easy to see that $ \beta (\frac{1}{2}, x) = \frac{1}{2}\beta(x),$ for any $x \in S_\mathbb{X}.$ Therefore, taking $\lambda = \frac{1}{2},$ we state the following result as a particular case of Theorem \ref{general} that studies the converse of \cite[Lemma 2.2(i)]{GL}.

 \begin{theorem}\label{local}
 Let $\mathbb{X} $ be  a strictly convex normed space and  let $x_0 \in S_{\mathbb{X}}.$ If  $ y_0 \in S_{\mathbb{X}}$ is such that $\beta(x_0) = \min\{ \|x_0-y_0\|, \|x_0+y_0\|\},$ then $x_0 \perp_{I} y_0.$ 
 \end{theorem}

The following example illustrates that the condition of  strict convexity in the above theorem cannot be relaxed in general.

 \begin{example}\label{strict convexity}
     Let $\mathbb{X} = \ell_\infty^2$ and let $x_0 = (1, 0) \in S_{\mathbb{X}}.$ To compute $\beta((1,0)),$ we observe that any $y \in S_{\ell_\infty^2}$ can be written as either $ y = (\alpha, \pm1)$ or $ y = (\pm1, \alpha), $ where $-1\leq \alpha \leq 1.$ It is straightforward to observe that whenever $y = (\alpha, \pm1),$ $\min\{\|x_0-y\|_\infty, \|x_0+y\|_\infty\} = 1.$ On the other hand, $\min\{\|x_0-y\|_\infty, \|x_0+y\|_\infty\} = |\alpha| \leq 1,$ when $y = (\pm1, \alpha).$ Therefore, $\beta((1, 0)) = 1.$ Clearly, for any $y = (\alpha, \pm 1)$ with $ 0 < \alpha < 1,$ we have that
     \[
     \min\{\|x_0-y\|_\infty, \|x_0+y\|_\infty\} = \min\{ 1, |1+\alpha|\} = 1 = \beta((1, 0)). 
     \]
     In particular, we observe that isosceles orthogonality is not a necessary condition for the attainment of $ \beta(x), $ where $ x \in S_{\mathbb{X}}. $
 \end{example}

  \begin{remark}
 For another local constant $ \alpha(x),$ introduced in \cite{GL}, using similar arguments as in  Theorem \ref{local},  we conclude that if $ x_0, y_0 \in S_{\mathbb{X}}$ with $ \max\{\|x_0 - y_0\|, \|x_0 + y_0\|\}  = \alpha(x_0)$ then $x_0 \perp_{I} y_0,$ provided $\mathbb{X}$ is strictly convex.
  \end{remark}

   Regarding the attainment of the local James constant $ \beta(x) $ in an arbitrary Banach space, we have already noticed that if there exist $x_0, y_0 \in S_{\mathbb{X}}$ such that $\min\{\|x_0 - y_0\|, \|x_0 + y_0\|\} = \beta(x_0)$ then $x_0 \not\perp_{I} y_0,$ in general. However, as illustrated in the following theorem, we are going to observe a stronger behavior with respect to isosceles orthogonality, in the case of attainment of the corresponding global constant $J(\mathbb{X}).$

 \begin{theorem}\label{global}
 Let $\mathbb{X}$ be a normed space. Let $u, v \in S_{\mathbb{X}}$ be such that $\min\{\|u-v\|, \|u+v\|\} = J(\mathbb{X}).$ Then $u\perp_{I} v.$
 \end{theorem}
 
 \begin{proof}
    We prove the theorem by considering the following two possible cases.
    
    \smallskip
   \textit{Case} $(i)$: Let us assume that $J(\mathbb{X}) = 2.$ Then $\min\{\|u-v\|, \|u+v\|\}= 2. $ It is trivial to see that $\max\{\|x+y\|, \|x-y\|: x,y \in S_{\mathbb{X}}\} \leq 2.$  Therefore, it necessarily follows that $\|u+ v\| = \|u-v\|,$ i.e., $u \perp_{I} v.$
   
   \smallskip
   \textit{Case} $(ii)$: Suppose that $J(\mathbb{X}) < 2.$ Consider the set $S = \{\epsilon \in [0, 2) : \epsilon < 2-2\delta_{\mathbb{X}}(\epsilon)\},$ where $\delta_{\mathbb{X}}(\epsilon)= \inf \{1-\frac{\|x+y\|}{2} : x, y \in S_{\mathbb{X}} ~ and ~ \|x-y\| \geq \epsilon\}.$ From Theorem \ref{sup}, we observe that $ \sup S = J(\mathbb{X}) < 2.$   Suppose on the contrary that $u \not \perp_{I} v.$ Without loss of generality, let us assume that $\|u+v\| > \|u - v\|.$ Also, let $\|u - v\| = \epsilon_0 = J(\mathbb{X}) <2 .$ Then, $1- \frac{\|u+v\|}{2} < 1- \frac{\epsilon_0}{2}$ implies that $\delta_{\mathbb{X}}(\epsilon_0) < 1 - \frac{\epsilon_0}{2},$ i.e., $\epsilon_0 < 2 - 2\delta_{\mathbb{X}}(\epsilon_0).$ Therefore, $\epsilon_0 \in S. $ Now, from \cite[Cor. 5]{L}, we note that  $\delta_{\mathbb{X}}(\epsilon)$ is a continuous function on $[0, 2).$ Therefore, it is easy to verify that $S$ is an open set in $ \mathbb{R}, $ with its usual topology.  Since $\epsilon_0 \in S$ and $ S $ is open, it follows that there exists $ \mu_0 > 0 $ such that $ \epsilon_0 + \mu_0 \in S, $ which contradicts our assumption that $ \sup S = J(\mathbb{X}) = \epsilon_0. $ Hence $\|u - v\| = \|u+ v\|,$ i.e., $u \perp_{I} v,$ as claimed.  
 \end{proof}
 
\begin{remark}\label{attainment}
 For $x \in S_\mathbb{X}$ and $\epsilon \in [0, 1),$ let us consider the set $A(x, \epsilon) = \{ y\in S_\mathbb{X} : x\perp_I^{\epsilon} y \}.$ Now it is easy to see that for any $\epsilon > 0,$ if $z \in \{\mathbb{X} \setminus A(x, \epsilon)\} \cap S_\mathbb{X}$ then $\min \{ \|x+z\|, \|x-z\| \} < J(\mathbb{X}).$ For, otherwise, from Theorem \ref{global} we obtain $x \perp_I z,$ which contradicts $z \in \{\mathbb{X} \setminus A(x, \epsilon)\} \cap S_\mathbb{X}.$  
\end{remark}

 In view of the example \ref{strict convexity}, it is natural to speculate whether strict convexity is essential for Theorem \ref{local}. We negate this by means of the following explicit example, constructed with the help of Theorem \ref{global}.\\
 Let us recall from \cite{KST} that for each $\theta \in \mathbb{R},$ the $\theta$-rotation matrix $R(\theta)$ is given by 
  \[ R(\theta) = \begin{pmatrix} \cos\theta & -\sin\theta\\
 	\sin\theta & \cos\theta               \end{pmatrix}.
 \] 
 A norm $\|. \|$ on $ \mathbb{R}^2$ is said to be $ \theta$-invariant if $R(\theta)$ is an isometry on $(\mathbb{R}^2, \| . \|).$ 
  
  \begin{example}\label{octagon}
       Let $\mathbb{X}$ be the two-dimensional Banach space, identified as $ \mathbb{R}^2, $ endowed with the norm $ \|(x, y)\| = \max\{|x|, |y|, 2^{-1/2} (|x| + |y|)\} $ for any $(x, y) \in \mathbb{R}^2.$ It is easy to verify that $S_\mathbb{X}$ is a regular octagon, with vertices $\pm v_1 = \pm (1, \sqrt{2}-1), \pm v_2= \pm (\sqrt{2}-1, 1), \pm v_3=\pm (1-\sqrt{2}, 1), \pm v_4= \pm (-1, \sqrt{2}-1).$ The unit sphere is shown in the following figure:\\

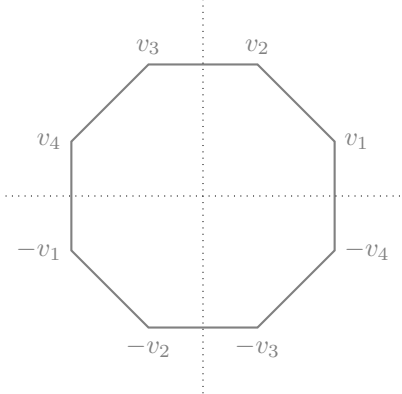
  \begin{center}     
 \begin{tikzpicture}[scale=1.75]
 \draw[gray, thick] (1,0.414) coordinate[label=right:$v_1$] ($v_1$)
 -- (0.414,1) coordinate[label=above:$v_2$] ($v_2$)
 -- (-0.414,1) coordinate[label=above:$v_3$] ($v_3$)
 -- (-1,0.414)coordinate[label=left:$v_4$] ($v_4$)
 -- (-1,-0.414)  coordinate[label=left:$-v_1$] ($-v_1$)
 -- (-0.414,-1)  coordinate[label=below:$-v_2$] ($-v_2$)
 -- (0.414,-1)  coordinate[label=below:$-v_3$] ($-v_3$) 
 -- (1,-0.414)  coordinate[label=right:$-v_4$] ($-v_4$)
 -- cycle;
 
 \draw[dotted] (-1.5,0) -- (1.5,0);
 \draw[dotted] (0,-1.5) -- (0,1.5);
  \end{tikzpicture}
  \end{center}

       It is easy to see that the given norm on $\mathbb{R}^2$ is $\frac{\pi}{4}$-invariant. Let $E_1$ be the edge joining the vertices $-v_4, v_1.$ Therefore, the following two conditions are equivalent.\\
       
       $(i)$ If for any $\widetilde{x}\in E_1$ there exists an $\widetilde{y} \in S_{\mathbb{X}}$ such that $\min\{\|\widetilde{x} - \widetilde{y}\|, \|\widetilde{x} + \widetilde{y}\|\} = \beta(\widetilde{x})$ then $\widetilde{x} \perp_{I} \widetilde{y}.$\\
       
       $(ii)$ If for any $\widetilde{x}\in S_{\mathbb{X}}$ there exists an $\widetilde{y} \in S_{\mathbb{X}}$ such that $\min\{\|\widetilde{x} - \widetilde{y}\|, \|\widetilde{x} + \widetilde{y}\|\} = \beta(\widetilde{x})$ then $\widetilde{x} \perp_{I} \widetilde{y}.$\\
       
       We will show that $(i)$ holds true.  Any $u\in E_1$ can be written as $ u = (1, \gamma),$ where $|\gamma|\leq \sqrt{2}-1.$  Also, given any $u = (1, \gamma) \in E_1,$ we have $v = \pm (-\gamma, 1)\in S_\mathbb{X}$ such that $ u \perp_I  v.$ From  Lemma \ref{lemma}, we obtain that $ \beta(v)= \| u - v\|= \sqrt{2}. $ On the other hand, from  Proposition \ref{equal} it follows that $J(\mathbb{X}) = \sqrt{2}.$ This implies that $ (u, v) \in M_{J(\mathbb{X})}.$\\
       Since $\beta(\widetilde{x}) = \sqrt{2}$ for any $\widetilde{x} \in E_1,$ it is easy to see that $\min\{ \|\widetilde{x}+\widetilde{y}\|, \|\widetilde{x}-\widetilde{y}\|\}= \beta (\widetilde{x})$ implies that $(\widetilde{x}, \widetilde{y})\in M_{J(\mathbb{X})}.$ Now applying Theorem \ref{global}, we conclude that $\widetilde{x} \perp_{I} \widetilde{y}.$\\
        In particular, Theorem \ref{local} may indeed hold true for certain Banach spaces which are not strictly convex.
  \end{example}

As a complementary notion of the James constant, we may also consider the Sch\"affer constant, in view of Theorem \ref{global}.
 It can be shown similarly by using the method from \cite[Th. 3.3]{GL} that:
 \[  S(\mathbb{X}) = \inf \{ \epsilon : \epsilon > 2 - 2\rho_{\mathbb{X}}(\epsilon)\}.
 \]
 Recall that  $ \rho_{\mathbb{X}}(\epsilon)$ is continuous  and convex (see \cite{WY})  on $[0, 2].$ Therefore, applying similar methods as used in the Theorem \ref{global} we obtain the following result.
 
 \begin{theorem}
    Let $\mathbb{X}$ be a normed space. Let $ u, v \in S_{\mathbb{X}}$ be such that $\min\{\|u - v\|, \|u + v\|\} = S(\mathbb{X}). $ Then $u \perp_{I} v.$
    
 \end{theorem}

  Next we  show that in a two-dimensional polyhedral Banach space, the James constant  is always attained at one of the extreme points of the unit ball. To achieve this we need the following lemma.
  
  \smallskip

 \begin{lemma}\label{orientation1}
  Let  $\mathbb{X}$ be a two-dimensional Banach space. Let $ v_1, v_2 \in S_\mathbb{X}$ be such that $ v_1 \not= \pm v_2$ and $v_1 \prec v_2.$ Suppose that $w_1, w_2 \in S_\mathbb{X}$ are such that $ v_i \perp_I w_i$ and $ -w_i \prec v_i \prec w_i, $ for $ i \in \{1, 2\}.$ Then $ w_1 \prec w_2.$

 \end{lemma}

 \begin{proof}
 
It follows from Theorem \ref{uniqueness} that $ w_1 \neq\pm w_2.$ Suppose on the contrary that $ w_1 \not\prec w_2.$ Then $ w_2 \prec w_1.$ Therefore, the only possibility is that $ v_1 \prec v_2 \prec w_2 \prec w_1 \prec -v_1.$ This implies that the ray $ [0, w_2\rangle$ lies in between the rays $[0, v_1\rangle$ and $[0, w_1\rangle$ and the ray $ [0, v_2\rangle$ lies in between the rays $[0, v_1\rangle$ and $[0, w_2\rangle.$ Now applying Lemma \ref{monotonicity} we get, 
 \[
   \|v_1 - w_2\| \leq \|v_1 - w_1\| = \|v_1 + w_1\| \leq \|v_1 + w_2\|,
 \]
 and
 \[
  \|v_1 - w_2 \| = \|w_2 - v_1\| \geq \|w_2 - v_2 \| = \|w_2 + v_2\| \geq \|w_2 + v_1\| = \|v_1 + w_1\|. 
 \]
  Thus  $ \| v_1 + w_2\| = \| v_1 - w_2\|,$ which shows that $v_1$ is  isosceles orthogonal to $w_2$. This is a contradiction as $ w_1 \neq \pm w_2.$ Therefore, $ w_1 \prec w_2,$ as desired. 
 \end{proof}
 
 The following important remark, which is immediate from the above lemma, is also relevant for the proof of our next theorem.
 
 \begin{remark}\label{orientation2}
 Let $\mathbb{X}$ be a two-dimensional Banach space. Let $ v_1, v_2 \in S_\mathbb{X}$ be such that $ v_1 \not= \pm v_2$ and let $w_1, w_2 \in S_\mathbb{X}$ be such that $ v_i \perp_I w_i$ and let $ -w_i \prec v_i \prec w_i, $ for $ i \in \{1, 2\}.$ Without loss of generality we can assume that $v_1 \prec v_2.$ Suppose $ v \in S_\mathbb{X}$ is such that the ray $ [0, v\rangle$ lies in between the rays $ [0, v_1\rangle$  and $[0, v_2\rangle,$  which implies that $v_1 \prec v \prec v_2.$ From Lemma \ref{orientation1}, it can be concluded that $ w_1 \prec w \prec w_2.$ In other words, the ray $[0, w\rangle$ lies in between the rays $ [0, w_1\rangle$ and $[0, w_2\rangle,$ where $ v \perp_I w.$ 
 \end{remark}
 
 We are now in a position to prove the following result.
 
 \begin{theorem}\label{polyhedron}
    
    Let $ \mathbb{X}$ be a two-dimensional polyhedral Banach space. Then there exists $z \in E_\mathbb{X}$ such that $\beta(z) = J(\mathbb{X}),$ i.e., $ \|z + y\| = \|z - y\| = J(\mathbb{X}),$ where $y \in S_{\mathbb{X}}$ and $ z \perp_{I} y.$ 
 \end{theorem}
 
 \begin{proof}
    
     Let  $v_1, v_2$ be two  extreme points of $ B_\mathbb{X}$  such that $ v_1 \prec v_2 $ and $ tv_1 + (1-t) v_2 \in S_{\mathbb{X}}, $ for all $t \in [0,1].$ Then there exist $ w_1, w_2 \in S_{\mathbb{X}}$ such that $ v_1 \perp_{I} w_1, v_2 \perp_I w_2$ and $ w_1 \prec w_2,$ by using Lemma \ref{orientation1}.  We  consider the  following two cases:
     
     \smallskip
     
     Case $1$ : At first we consider that $ \lambda w_1 + (1 - \lambda)w_2 \subset S_\mathbb{X},$ for all $ \lambda \in [0, 1].$ For any $v \in [v_1, v_2],$ we can write $v = t_0v_1 + (1-t_0)v_2,$ for some $t_0 \in [0, 1].$ Take $w \in S_\mathbb{X}$ such that $ v \perp_I w.$ By virtue of Remark \ref{orientation2}, it follows that the ray $ [0, w\rangle $ lies in between the rays $ [0, w_1\rangle$ and $[0, w_2\rangle.$ Now, if $ w = t_0w_1+ (1- t_0)w_2,$ then using Lemma  \ref{lemma}, we get
     \begin{eqnarray*}
          \beta(v) &=& \| v - w\|\\
                 &=& \|t_0v_1 + (1-t_0)v_2 - t_0w_1 - (1-t_0)w_2\|\\
                & \leq& t_0\|v_1-w_1\| + (1-t_0)\|v_2-w_2\|\\
                 &=& t_0 \beta(v_1) + (1-t_0) \beta(v_2).
     \end{eqnarray*}
     If $ w \not= t_0w_1+ (1- t_0)w_2$ then either $\|v - w\| \leq \| v- (t_0w_1+ (1- t_0)w_2)\|$ or $\|v - w\| > \|v - (t_0w_1+ (1- t_0)w_2) \|.$ Applying Lemma \ref{monotonicity}, it is straightforward to see that in the latter case we have $\|v + w\| \leq \| v + t_0w_1+ (1- t_0)w_2\|.$ Therefore, we get, 
     \begin{eqnarray*}
        \beta(v) &=& \| v \pm w\|\\
                 &=& \|t_0v_1 + (1-t_0)v_2 \pm w\|\\
                 &\leq& \|t_0v_1 + (1-t_0)v_2 \pm (t_0w_1 - (1-t_0)w_2)\|\\
                & \leq& t_0\|v_1\pm w_1\| + (1-t_0)\|v_2 \pm w_2\|\\
                 &=& t_0 \beta(v_1) + (1-t_0) \beta(v_2).
      \end{eqnarray*}
      Therefore,   $ \beta(v) \leq \max\{ \beta(v_1), \beta(v_2)\}.$
      
      \smallskip
      
       Case $2$ :  Let $\{\lambda w_1 + (1 -\lambda)w_2: \lambda \in [0,1] \}\not \subset S_\mathbb{X}.$ Assume that there exist $k$ extreme points $x_1, x_2, \ldots, x_k$ lying in between the rays $ [0, w_1\rangle$ and $[0, w_2\rangle$ such that $w_1\prec x_1 \prec  x_2\prec  \ldots\prec  x_k\prec w_2 . $ Then following Remark \ref{orientation2}, we get $ z_1, z_2, \ldots, z_k \in [v_1, v_2]$ such that $ v_1\prec z_1\prec z_2 \prec \ldots \prec z_k \prec v_2$  and $ z_i \perp_{I} x_i,$ for $ 1 \leq i \leq k. $ Considering the  segments $[v_1, z_1],$ $ [z_1, z_2], \ldots, [z_k, v_2] $ in place of $[v_1, v_2]$ and applying similar arguments as in  Case $1,$ we get,  for any $v \in [v_1, v_2],$ 
 \begin{eqnarray*}
    \beta(v) &\leq& \max\{ \beta(v_1), \beta(z_1), \ldots, \beta(z_k), \beta(v_2) \}\\
           &=& \max\{ \beta(v_1), \beta(x_1), \ldots, \beta(x_k), \beta(v_2)\}.
  \end{eqnarray*}
         Therefore, we observe that for any $ v \in [v_1, v_2],$ there exists $z \in E_\mathbb{X}$ such that $\beta(v) \leq \beta(z).$ As $v_1, v_2$ are chosen arbitrarily, we can conclude that for any $ v \in S_\mathbb{X}$ there exists $z \in E_\mathbb{X}$ such that $ \beta(v) \leq \beta(z).$ This completes the proof of the theorem.
     
\end{proof}

The following remark is immediate from  Theorem \ref{polyhedron}. 
\begin{remark}\label{remark}
 Let $\mathbb{X}$ be a two-dimensional polyhedral Banach space. Suppose that $\pm v_1, \pm v_2, \ldots, \pm v_m$ are the extreme points of $B_\mathbb{X}.$ From Theorem \ref{polyhedron}, it can be easily seen that to find the James constant $ J(\mathbb{X}), $ we only need to deal with the extreme points of the unit ball of $\mathbb{X}.$  Indeed, we can compute the James constant $ J(\mathbb{X}) $ in a more efficient way by the formula:
\[
 J(\mathbb{X}) = \max_{1 \leq i \leq m} \beta(v_i) = \max\{ \| v_i + w_i \| : 1 \leq i \leq m, w_i \in S_\mathbb{X} ~~ and~ ~v_i \perp_I w_i\}.
 \]
 
\end{remark}

 In the following example, we will show the applicability of Theorem \ref{polyhedron} towards explicitly computing the James constant, as described in Remark \ref{remark}. 

\begin{example}\label{irregular}
     
  Consider a two-dimensional polyhedral Banach space $\mathbb{X}$ whose unit sphere is an irregular hexagon, as shown in the following figure:\\

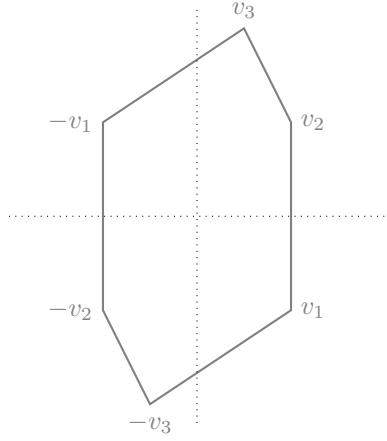
\begin{center}
 \begin{tikzpicture}[scale=1.25]
 
 \draw[gray, thick] (1,-1) coordinate[label=right:$v_1$] ($v_1$)
 -- (1,1) coordinate[label=right:$v_2$] ($v_2$)
 -- (1/2,2) coordinate[label=above:$v_3$] ($v_3$)
 -- (-1,1)coordinate[label=left:$-v_1$] ($-v_1$)
 -- (-1,-1)  coordinate[label=left:$-v_2$] ($-v_2$)
 -- (-1/2,-2)  coordinate[label=below:$-v_3$] ($-v_3$)
 -- cycle;
 
 \draw[dotted] (-2,0) -- (2,0);
 \draw[dotted] (0,-2.2) -- (0,2.2);
 
  \end{tikzpicture}
 \end{center}

      The vertices of $B_\mathbb{X}$ are $ \pm v_1=\pm (1, -1),\pm v_2= \pm(1, 1), \pm v_3 = \pm (\frac{1}{2}, 2).$ Clearly,  $\beta(x) = \beta(-x),$ for any $x \in \mathbb{X},$ so that we only need to calculate  $\beta(1, -1), \beta(1, 1), $ $ \beta(\frac{1}{2}, 2).$ By a straightforward computation, we have $ (1, -1) \perp_I \pm(\frac{9}{13}, \frac{21}{13}),$ $(1, 1) \perp_I \pm(-\frac{5}{17}, \frac{25}{17})$ and $(\frac{1}{2}, 2) \perp_I \pm(1, -\frac{2}{7}).$ Using \textcolor{blue}{Lemma \ref{lemma}}, we obtain that
      \begin{eqnarray*}
       \beta(1, -1) &= &\| (1, -1) + (\frac{9}{13}, \frac{21}{13})\| = \| (\frac{22}{13}, \frac{8}{13}) \| =  \frac{22}{13},\\
       \beta(1, 1) &=& \| (1, 1) + (-\frac{5}{17}, \frac{25}{17})\| = \| (\frac{12}{17}, \frac{42}{17})\| = \frac{22}{17},\\
       \beta(\frac{1}{2}, 2) &=& \|(\frac{1}{2}, 2) + (1, -\frac{2}{7}) = \|(\frac{3}{2}, \frac{12}{7})\| = \frac{11}{7}.
      \end{eqnarray*}
     Thus  $ J(\mathbb{X}) = \max \{ \frac{22}{13}, \frac{22}{17}, \frac{11}{7}\} = \frac{22}{13}.$

\end{example}
 The above example illustrates that the problem of finding the James constant in a two-dimensional polyhedral Banach space $\mathbb{X} $ is equivalent to calculating the local constant $\beta(x)$ only for the finitely many extreme points of $ B_{\mathbb{X}}.$
\smallskip

 In the rest of the article, we study approximate isosceles orthogonality and its role in the attainment of the modulus of convexity, an important geometric constant associated with a given normed space. 
   We begin with the following basic observation.

\begin{prop}

Let $ \mathbb{X} $ be a normed space and let $x, y \in S_{\mathbb{X}} $ with $x \neq \pm y.$ Then there exists an $\epsilon \in [0, 1)$ such that $ x \perp_{I}^{\epsilon} y.$ 

\end{prop}

 \begin{proof}
  If $ x \perp_I y $ then we are done by taking $ \epsilon = 0. $ Suppose that $ x \not\perp_I y.$ Since $x \neq \pm y,$ it follows that $ |\|x+y\|^2 - \| x-y \|^2| = 4 - \epsilon_0,$ for some $ 0<\epsilon_0<4. $ Therefore, choosing $ \epsilon \in [\frac{4-\epsilon_0}{4}, 1) $ we conclude that $ |\| x+ y\|^2 - \|x - y \|^2 |\leq 4\epsilon,$ i.e., $ x \perp_I^{\epsilon} y. $  
 \end{proof}

 For a given $\epsilon \in [0,2],$ let us define the set:
 \[
  M_{\delta_{\mathbb{X}}(\epsilon)} = \Big \{ (x, y) \in S_\mathbb{X} \times S_\mathbb{X} : 1- \frac{\|x +y\|}{2} = \delta_{\mathbb{X}}(\epsilon) \Big\}.
 \] $  M_{\delta_{\mathbb{X}}(\epsilon)} $ is called the attainment set of $\delta_{\mathbb{X}}(\epsilon),$ for any $\epsilon \in [0, 2].$ It is clear that whenever $\mathbb{X}$ is finite-dimensional, $M_{\delta_{\mathbb{X}}(\epsilon)} \neq \emptyset.$ Our next result shows that the attainment of $\delta_{\mathbb{X}}(\epsilon)$ is closely related to approximate isosceles orthogonality.
 
 \begin{theorem}\label{modulus}
 Let $\mathbb{X}$ be a normed space. Let $ M_{\delta_{\mathbb{X}}(\epsilon)} \neq \emptyset,$ for some $\epsilon \in (0, 2).$ Then there exists $(u_0, v_0) \in M_{\delta_{\mathbb{X}}(\epsilon)}$ such that $ u_0 \perp_I^{\epsilon_0} v_0,$ where 
 $\epsilon_0 = |1+ \delta_{\mathbb{X}}(\epsilon)^2 - 2\delta_{\mathbb{X}}(\epsilon)- \frac{\epsilon^2}{4}|\in [0, 1).$
 \end{theorem}
 
 \begin{proof}
 
 Suppose that $(u, v) \in M_{\delta_{\mathbb{X}}(\epsilon)}.$ Since $\epsilon \in (0, 2),$ it is clear that $ u \neq \pm v.$ Consider the set $ P_u = \{ w \in S_\mathbb{X} : \| u - w\|= \epsilon\}.$ We claim that there exists $w' \in P_u$ such that $(u, w') \in M_{\delta_{\mathbb{X}}(\epsilon)}.$ If $v \in P_u$ then our claim holds true. Let us now assume that $v \notin P_u.$ Suppose on the contrary that the claim is not true. Then clearly, $ \delta_{\mathbb{X}}(\epsilon) < 1 - \frac{\|u+w\|}{2}$ for all $w \in P_u,$ i.e., $ \|u+v\| > \|u + w\|.$ Considering the two-dimensional subspace $\mathbb{Y} = span\{u, v\}$ and applying Lemma \ref{monotonicity}, we obtain that $\|u - v\| \leq \|u -w\|$ for all $ w \in P_u.$ As $v \notin P_u,$ we have $\|u-v\| < \|u-w\|$ for all $w \in P_u,$ which is a contradiction to the fact that $ \|u-v\| \geq \epsilon.$ This establishes our claim.
  It is now easy to observe that there exists $(u_0, v_0) \in M_{\delta_{\mathbb{X}}(\epsilon)}$ such that $\|u_0 -v_0\| = \epsilon.$ This implies that $|\|u_0 + v_0\|^2 - \|u_0 - v_0\|^2| = 4 |1+ \delta_{\mathbb{X}}(\epsilon)^2 - 2\delta_{\mathbb{X}}(\epsilon)- \frac{\epsilon^2}{4}|.$ Let $\epsilon_0 = |1+ \delta_{\mathbb{X}}(\epsilon)^2 - 2\delta_{\mathbb{X}}(\epsilon)- \frac{\epsilon^2}{4}|.$ Then $ 0 \leq \epsilon_0 < 1$ and  $|\|u_0 + v_0\|^2 - \|u_0 - v_0\|^2| = 4\epsilon_0,$ which shows that $ u_0 \perp_I^{\epsilon_0} v_0.$

 \end{proof}

  In case $\mathbb{X}$ is strictly convex, we have the following corollary to the above theorem.
 \begin{cor}
 Let $\mathbb{X}$ be a strictly convex normed space and let $ \epsilon \in (0,2).$  If $(u, v) \in M_{\delta_{\mathbb{X}}(\epsilon)}$ then $u \perp_{I}^{\epsilon_0} v ,$ where $\epsilon_0 = |1+ \delta_{\mathbb{X}}(\epsilon)^2 - 2\delta_{\mathbb{X}}(\epsilon)- \frac{\epsilon^2}{4}|\in [0, 1).$ 
 \end{cor}
 
 \begin{proof}
 
  Given $ \epsilon \in (0, 2),$ we only need to show that for any $(u, v) \in M_{\delta_{\mathbb{X}}(\epsilon)},$ it necessarily follows that $ \|u-v\|= \epsilon.$  Suppose on the contrary that $ \|u-v\| > \epsilon.$ Consider the set $ P_u = \{ w \in S _{\mathbb{X}} : \|u-w \| = \epsilon\}.$ Clearly, $ v \not\in P_u$  and for any $ w \in P_u,$ we have that $ \| u-v \| > \|u-w\|.$ Therefore, by Lemma \ref{monotonicity}, together with strict convexity, we get $ \|u+v\| < \|u+w\|$ and so $ 1 - \frac{1}{2}\|u+v\| > 1 - \frac{1}{2}\|u+w\|,$ which contradicts the fact that $\delta_{\mathbb{X}}(\epsilon) = 1-\frac{\|u+v\|}{2}.$ Now proceeding similarly as in the proof of Theorem \ref{modulus}, we obtain the desired conclusion.
 
 \end{proof}
 
 In connection with the explicit computation of $ \delta_{\mathbb{X}}(\epsilon), $ the following remark seems relevant.
 
 \begin{remark}
 
 For $ \epsilon \in (0, 2),$ let us consider the set :
 \[ G_\epsilon = \{ (u, v) \in S_{\mathbb{X}} \times S_{\mathbb{X}} : u \perp_{I}^{\epsilon_0} v ~ and ~ \|u - v\| = \epsilon\}, \]
 where $\epsilon_0 =|1+ \delta_{\mathbb{X}}(\epsilon)^2 - 2\delta_{\mathbb{X}}(\epsilon)- \frac{\epsilon^2}{4}|. $ Clearly, $ G_\epsilon$ is a closed set with respect to the usual product topology defined on $ \mathbb{X} \times \mathbb{X}.$ It can be readily seen that whenever $\mathbb{X}$ is finite-dimensional, there exists $ (u_1, v_1) \in G_\epsilon$ such that $ \delta_{\mathbb{X}}(\epsilon) = 1 - \frac{\|u_1 + v_1\|}{2}. $ Therefore, we conclude that to find the value of $ \delta_{\mathbb{X}}(\epsilon), $ for any $ \epsilon \in (0, 2),$ we only need to take into account the subset $ G_\epsilon.$    
 \end{remark}

In \cite{SPM}, the authors explored the geometric structure of the approximate Birkhoff-James orthogonality set. Motivated by this, we study the same in the case of approximate isosceles orthogonality, in our next theorem. For this purpose, we consider the $ \epsilon $-approximate isosceles orthogonality set $ A(x, \epsilon), $ corresponding to the vector $ x \in S_{\mathbb{X}} $ and $ \epsilon \in [0, 1), $ as defined in Remark \ref{attainment}:\\
  We end the present article with the following characterization of $ A(x, \epsilon). $  
 
\begin{theorem}
   Let $\mathbb{X}$ be a two-dimensional Banach space. Then for any $x \in S_\mathbb{X},$ $A(x, \epsilon) = D \cup -D,$ where $D$ is a connected subset of $S_{\mathbb{X}}.$ 
\end{theorem}

\begin{proof}
 We note from  Theorem \ref{uniqueness} that for $x \in S_{\mathbb{X}},$ there exists a unique (except for the sign) $y \in S_{\mathbb{X}}$ such that $x \perp_{I} y.$ For each $t \in [0,1],$  let $u_t =\frac{(1-t)x+ty}{\|(1-t)x+ty\|} $ and $v_t=  \frac{-(1-t)x+ty}{\|-(1-t)x+ty\|}.$ Consider the sets $
       R = \{ t\in [0, 1] : x \perp_{I}^{\epsilon}u_t\}, $ and $
      L = \{ t\in [0, 1] : x \perp_{I}^{\epsilon}v_t\}.$
   Clearly, $R, L \neq \emptyset,$ since $1 \in R\cap L.$ Next we prove that $R$ and $L $ are closed. Suppose $\{t_n\}_{n\in \mathbb{N}} \in R$ is such that $t_n \rightarrow t.$  Then $x \perp_{I}^{\epsilon} u_{t_n} . $ This implies that for every $ n \in \mathbb{N},$ we have $|\|x+ u_{t_n}\|^2 - \|x- u_{t_n}\|^2| \leq 4\epsilon.$    As $n \rightarrow \infty,$ $|\|x+ u_{t}\|^2 -\|x -  u_{t}\|^2| \leq 4\epsilon.$ Therefore, $x \perp_{I}^{\epsilon} u_t.$ This proves that $R$ is closed. Similarly, it can be shown that $L$ is also closed.\\
   Let $t_R = \inf R$ and let $t_L = \inf L.$ Then using Lemma \ref{monotonicity}, for any $t\in [0, 1]$ with $t\geq t_R,$ we get that 
    $ \|x - u_t\| \geq \|x - u_{t_R}\|$ and $   \|x + u_t\| \leq \|x + u_{t_R}\|.$
     This gives 
     $|\|x + u_t\|^2-  \|x - u_t\|^2| \leq |\|x + u_{t_R}\|^2 -\|x - u_{t_R}\|^2| \leq 4\epsilon. $ Therefore, $x \perp_{I}^{\epsilon} u_t. $ Similarly, one can show that for any $t \in [0,1]$ with $t \geq t_L,$ $x \perp_{I}^{\epsilon} v_t.$ Consider  $D = \Big \{ \frac{s u_{t_R} + (1-s)u_{t_L}}{\|s u_{t_R} + (1-s)u_{t_L}\|} : 0 \leq s \leq 1 \Big \}.$ Clearly, $D$ is connected. Moreover, it is easy to see that $ D \cup (-D) \subset A(x, \epsilon). $ Also, the implication $ A(x, \epsilon) \subset D \cup (-D) $ is trivial from the description of the sets $ R $ and $ L. $ This completes the proof of the theorem.

\end{proof}

 \subsection*{Declarations}
 Data sharing is not applicable to this article as no datasets were generated or analysed during the current study.
 Authors also declare that there is no financial or non-financial interests that are directly or indirectly related to the work submitted for publication.  On behalf of all authors, the corresponding author states that there is no conflict of interest.

\end{document}